\documentclass[12pt]{amsart}
\usepackage[alphabetic]{amsrefs}
\usepackage{mathdots, blkarray}
\usepackage{multicol}
\usepackage{graphicx}
\usepackage{amscd}
\usepackage{upgreek}
\usepackage{stmaryrd}
\usepackage{longtable}
\usepackage[T1]{fontenc}
\usepackage{latexsym, amsmath, amssymb, amsthm}
\usepackage[svgnames]{xcolor}
\usepackage{rsfso}
\usepackage[mathscr]{eucal}
\usepackage{mathtools}
\usepackage{mathptmx}
\usepackage{titletoc}
\usepackage{wrapfig}
\usepackage{float}
\usepackage{xypic}
\usepackage{microtype}
\usepackage{dsfont}
\usepackage{xcolor}
\usepackage{color}
\usepackage[colorlinks = true,
            linkcolor  = DarkBlue,
            urlcolor   = DarkRed,
            citecolor  = DarkGreen]{hyperref}
\usepackage{hyperref}
\allowdisplaybreaks	
\usepackage[centering, includeheadfoot, hmargin=1.0in, tmargin=0.5in, 
  bmargin=0.6in, headheight=6pt]{geometry}
\newtheorem{theorem}{Theorem}[section]
\newtheorem{prop}[theorem]{Proposition}

\newtheorem{lem}[theorem]{Lemma}
\newtheorem{coro}[theorem]{Corollary}

\newtheorem{thm}[theorem]{Theorem}

\newtheorem{rem}[theorem]{\rm\textsc{Remark}}
\newtheorem{exam}[theorem]{\rm\textsc{Example}}
\newtheorem{conv}[theorem]{Convention}

\newcommand{\cv}[2]{\ensuremath{\left[\begin{smallmatrix} #1\\#2\end{smallmatrix}\right]}}
\newcommand{\cvv}[3]{\ensuremath{\left[\begin{smallmatrix} #1\\#2\\#3\end{smallmatrix}\right]}}

\DeclareMathOperator{\ad}{ad}
\DeclareMathOperator{\GL}{GL}
\DeclareMathOperator{\cha}{char}
\DeclareMathOperator{\dia}{diag}

\newcommand{\M}{\mathcal{M}} 
 
\newcommand{\C}{\mathbb{C}} 
\newcommand{\PP}{\mathbb{P}} 
\newcommand{\F}{\mathbb{F}} 
 
\newcommand{\N}{\mathbb{N}} 
\newcommand{\R}{\mathbb{R}} 
\newcommand{\g}{\mathfrak{g}} 

\newcommand{\ra}{\longrightarrow}

\newcommand{\wt}{\widetilde}
 
\newcommand{\hbo}{$\hfill\Diamond$} 

\begin{document}
\title{Some Lie Algebra Structures on Symmetric Powers} 
\def\shorttitle{Some Lie Algebra Structures on Symmetric Powers}

\author{Yin Chen}
\address{Faculty of Sciences, Algoma University, Brampton, ON, Canada, L6V 1A3 \& (Current address) Department of Finance and Management Science, University of Saskatchewan, Saskatoon, SK, Canada, S7N 5A7}
\email{yin.chen@usask.ca}

\begin{abstract}
Let $k$ be a field of any characteristic, $V$ a finite-dimensional vector space over $k$, and $S^d(V^*)$
be the $d$-th symmetric power of the dual space $V^*$. Given a linear map $\varphi$ on $V$ and an eigenvector $w$ of $\varphi$, we prove that the pair $(\varphi, w)$ can be used to construct a new Lie algebra structure on $S^d(V^*)$. We prove that this Lie algebra structure is solvable, and in particular, it is nilpotent if $\varphi$ is a nilpotent map. We also classify the Lie algebras for all possible pairs $(\varphi, w)$, when $k=\mathbb{C}$ and $V$ is two-dimensional. 
\end{abstract}

\date{\today}
\maketitle \baselineskip=16pt

\dottedcontents{section}[1.16cm]{}{1.8em}{5pt}
\dottedcontents{subsection}[2.00cm]{}{2.7em}{5pt}

\section{Introduction}
\setcounter{equation}{0}
\renewcommand{\theequation}
{1.\arabic{equation}}
\setcounter{theorem}{0}
\renewcommand{\thetheorem}
{1.\arabic{theorem}}

\noindent Fix a finite-dimensional vector space $V$ over a field $k$ and consider the variety of certain algebraic structures on $V$ such as linear transformations, bilinear forms, algebras, modules, or special representations, etc. The fundamental tasks include
classifying such algebraic structures up to an equivalence relation that usually needs to be realized by a group action, and finding a way to construct a set of representatives for the equivalence classes. Polynomial invariant theory that dates back to Hilbert's famous 14th problem and Noether's celebrated work plays a crucial role in addressing the first task; see \cite{DK15} or \cite{CW11} for general references of invariant theory. Specifically, invariant polynomials under a group action can be used to separate the equivalence classes while the difficulty in computing these separating invariants depends on the number of parameters used to describe such algebraic structures; see for example \cite{CCSW19, Che18, CSW21}, and \cite{CDG20}. Hence, reducing the number of parameters is the core part of applying invariant theory in classifying algebraic structures. 

Lie algebras are a kind of significant algebraic structure, establishing connections among various fields in geometry and algebra such as Lie groups and algebraic groups, and so they take a central position not only in pure mathematics but also in mathematical physics. However, parameterizing and classifying all Lie algebra structures on an $n$-dimensional vector space $V$ up to isomorphism is generally considered an impossible task as the number of parameters needed to describe all Lie algebras on $V$ usually depends on the exponents of the dimension $n$.  A well-known method to understand the collection of all Lie algebras on $V$ is to view each Lie algebra structure as a point in
the affine space of dimension $n^3$ via the structure constants of  Lie algebras; see, for example, \cite[Chapter 5]{GK96} or \cite{CCZ21,CZ23,DZ04} for other algebras in low-dimensional cases. 
In practice, this approach becomes intricate and appears to be inefficient when $n$ increases. Hence, as the value of $n$ becomes large, it is advisable to concentrate on subvarieties of specific Lie algebras on $V$ and identify fewer parameters to describe these subvarieties.

The subvariety of solvable (or nilpotent) Lie algebras presents a formidable challenge due to its similarities with the subvariety of finite solvable (or nilpotent) groups. Conversely, experiences in group theory suggest that achieving a complete classification of all finite $p$-groups up to isomorphism is impossible; see for example \cite[Chapter 1]{BNV07}. 
Therefore, constructing and classifying specific families of solvable Lie algebras have been both intriguing and crucial in the theory of Lie algebras.

The recent article \cite{DJ23} uses linear maps on a real finite-dimensional vector space $V$ and their eigenvectors to 
construct a family of solvable Lie algebra structures on the dual space $V^*$, showing that all three and four-dimensional real solvable Lie algebras can be realized via the approach they provide. One direct advantage of their approach is that those new solvable Lie algebras on $V^*$ can be parameterized by a matrix, a scalar (eigenvalue), and an eigenvector of the matrix, which means that one only needs at most $n^2+n+1$ parameters to describe these solvable Lie algebras. Recall that the dual space $V^*$ is isomorphic to the first symmetric power $S^1(V^*)$ of $V^*$ as vector spaces. It is reasonable to ask whether the construction method outlined in \cite{DJ23} can be extended to higher symmetric powers of $V^*$, and if successful, understanding or classifying these Lie algebra structures poses an interesting challenge. 

The purpose of this article is to answer these two questions and generalize some results in \cite{DJ23}  in terms of the following several directions. First of all, unlike \cite{DJ23}, the ground field $k$ we work over is of any characteristic and we do not restrict it to $\R$ or $\C$, except for the last section. Secondly, we show in the next section that the construction method outlined in \cite{DJ23} can be extended to all symmetric powers $S^d(V^*)$ of $V^*$ for all $d\in\N^+$, and furthermore, we prove that all such Lie algebra structures on $S^d(V^*)$ are solvable (Theorem \ref{thm1}). As an application, such an approach can be used to construct a new infinite dimensional Lie algebra structure on the symmetric algebra $S(V^*)$ of $V^*$ (Corollary \ref{coro2}). 
 Section 3 contains two examples that precisely illustrate the constructing method. In Section 4, we give a sufficient condition, 
independent of the degrees $d$ of symmetric powers, for when two Lie algebra structures obtained by such constructions are isomorphic  (Theorem \ref{thm2}). That allows us to use the language of group actions and orbits to understand the subvariety of these Lie algebras. Finally, we study the particular case where $k=\C$ and $V$ is two-dimensional, and classify those Lie algebra structures on any symmetric power $S^d(V^*)$ (Theorem \ref{thm3}). 

\section{Basic Constructions}
\setcounter{equation}{0}
\renewcommand{\theequation}
{2.\arabic{equation}}
\setcounter{theorem}{0}
\renewcommand{\thetheorem}
{2.\arabic{theorem}}

\noindent Let $V$ be an $n$-dimensional vector space over a field $k$ of any characteristic,  $d\in\N$ a nonnegative integer,  and $S^d(V^*)$ be the $d$-th symmetric power of the dual space $V^*$. Given a basis $\{e_1,e_2,\dots,e_n\}$ for $V$ and the basis $\{x_1,x_2,\dots,x_n\}$ for $V^*$ dual to $\{e_1,e_2,\dots,e_n\}$, we may identify $S^d(V^*)$ with the subspace of all homogeneous polynomials of degree $d$ of the polynomial ring $k[x_1,x_2,\dots,x_n]$. Thus elements in $S^d(V^*)$ can be viewed as homogenous functions on $V$ by evaluation. Note that $$\dim(S^d(V^*))={n+d-1\choose d}.$$

Let $\varphi:V\ra V$ be a linear map. We write $\varphi_d^*$ for the linear map on $S^d(V^*)$ induced by $\varphi$, i.e., $\varphi_d^*: S^d(V^*)\ra S^d(V^*)$ defined by $f\mapsto  \varphi_d^* (f)$ where $\varphi_d^* (f)(v):=f(\varphi(v))$ for all $v\in V$. 
Suppose $w$ denotes an eigenvector of $\varphi$ associated with an eigenvalue $\lambda$.
We define a bracket product on $S^d(V^*)$ (depending on $\varphi$ and $w$) as follows:
\begin{equation}
\label{bracket}
[f,g]:=g(w)\varphi_d^*(f)-f(w)\varphi_d^*(g)
\end{equation}
for all $f,g\in S^d(V^*)$. 

The first main result is the following theorem, which generalizes \cite[Theorems 1, 3 and 4]{DJ23}.

\begin{thm}\label{thm1}
With the bracket product (\ref{bracket}) above, $S^d(V^*)$ is a solvable Lie algebra. Moreover, if $\varphi$ is
nilpotent, then $S^d(V^*)$ is a nilpotent Lie algebra.
\end{thm}

\begin{proof}
Note that $\varphi_d^*$ is a linear map, so the bilinearity and antisymmetry of the bracket product could be verified immediately.
To see that   $S^d(V^*)$ is a Lie algebra, it suffices to verify the Jacobi identity. Given $f,g,h\in S^d(V^*)$, since $\varphi(w)=\lambda \cdot w$, it follows that
\begin{eqnarray*}
[[f,g],h] & = & [g(w)\varphi_d^*(f)-f(w)\varphi_d^*(g),h] =g(w) [\varphi_d^*(f),h] -f(w)[\varphi_d^*(g),h] \\
 &=&g(w)\Big(h(w)(\varphi_d^*)^2(f)-\varphi_d^*(f)(w)\varphi_d^*(h)\Big)-
 f(w)\Big(h(w)(\varphi_d^*)^2(g)-\varphi_d^*(g)(w)\varphi_d^*(h)\Big)\\
 &=&g(w)\Big(h(w)(\varphi_d^*)^2(f)-f(\varphi(w))\varphi_d^*(h)\Big)-
 f(w)\Big(h(w)(\varphi_d^*)^2(g)-g(\varphi(w))\varphi_d^*(h)\Big)\\
 &=&g(w)\Big(h(w)(\varphi_d^*)^2(f)-\lambda^df(w)\varphi_d^*(h)\Big)-
 f(w)\Big(h(w)(\varphi_d^*)^2(g)-\lambda^dg(w)\varphi_d^*(h)\Big)\\
 &=&\left(g(w)(\varphi_d^*)^2(f)- f(w)(\varphi_d^*)^2(g)\right)h(w).
\end{eqnarray*}
This fact also can be rewritten as the following form: 
\begin{equation}\label{key1}
\ad_h([f,g])=h(w)\cdot \Big(f(w)(\varphi_d^*)^2(g)-g(w)(\varphi_d^*)^2(f)\Big).
\end{equation}
Moreover, a direct computation shows that $[[f,g],h]+[[g,h],f]+[[h,f],g]=0$. Hence, 
$S^d(V^*)$ is a Lie algebra. To show that $S^d(V^*)$ is solvable, we consider arbitrary elements $f,g,h,\ell\in S^d(V^*)$ and note that $[f,g](w)=g(w)\varphi_d^*(f)(w)-f(w)\varphi_d^*(g)(w)=g(w)f(\varphi(w))-f(w)g(\varphi(w))=g(w)\cdot \lambda^d\cdot f(w)-f(w)\cdot \lambda^d\cdot g(w)=0.$ Hence, it follows from (\ref{key1}) that
\begin{eqnarray*}
[[f,g],[h,\ell]] & = &\ad_{[f,g]}([h,\ell]) =[f,g](w)\cdot\left(h(w)(\varphi_d^*)^2(\ell)- \ell(w)(\varphi_d^*)^2(f)\right) \\
& = &0\cdot \left(h(w)(\varphi_d^*)^2(\ell)- \ell(w)(\varphi_d^*)^2(f)\right)=0
\end{eqnarray*}
which means that $S^d(V^*)^{(2)}=\{0\}$. Therefore, $S^d(V^*)$ is a solvable Lie algebra.

To show the second statement, since $\varphi$ is nilpotent, we may choose a basis $\{e_1,e_2,\dots,e_n\}$ for $V$ such that
the matrix of $\varphi$ with respect to this basis is lower triangular with diagonals $0$. Thus the matrices  of $\varphi$
with respect to the dual basis $\{x_1,x_2,\dots,x_n\}$ of $V^*$ and the induced basis of $S^d(V^*)$  are both upper triangular with diagonals $0$. Hence, we may suppose $(\varphi_d^*)^m=0$ for some $m\in\N^+$. Consider arbitrary $s+2$ elements $f,g,h_1,h_2,\dots,h_{s}\in S^d(V^*)$. By (\ref{key1}), we see that
$$[h_1,[f,g]]=h_1(w)\cdot \Big(f(w)(\varphi_d^*)^2(g)-g(w)(\varphi_d^*)^2(f)\Big)$$
and 
\begin{eqnarray*}
[h_2,[h_1,[f,g]]] & = &h_1(w)\cdot  [h_2,f(w)(\varphi_d^*)^2(g)-g(w)(\varphi_d^*)^2(f)] \\
 & = & h_1(w)\cdot\Big(f(w) [h_2,(\varphi_d^*)^2(g)]-g(w)[h_2,(\varphi_d^*)^2(f)]\Big).
\end{eqnarray*}
As $[h_2,(\varphi_d^*)^2(g)]=(\varphi_d^*)^2(g)(w)\cdot \varphi_d^*(h_2)-h_2(w)\cdot (\varphi_d^*)^3(g)=
\lambda^{2d}\cdot g(w)\cdot \varphi_d^*(h_2)-h_2(w)\cdot (\varphi_d^*)^3(g)$ and 
$[h_2,(\varphi_d^*)^2(f)]=\lambda^{2d}\cdot f(w)\cdot \varphi_d^*(h_2)-h_2(w)\cdot (\varphi_d^*)^3(f)$, we see that
\begin{eqnarray*}
[h_2,[h_1,[f,g]]] & = & h_1(w)\cdot\Big(f(w) [h_2,(\varphi_d^*)^2(g)]-g(w)[h_2,(\varphi_d^*)^2(f)]\Big)\\
&=&(h_2\cdot h_1\cdot g)(w)\cdot (\varphi_d^*)^3(f)-(h_2\cdot h_1\cdot f)(w)\cdot (\varphi_d^*)^3(g).
\end{eqnarray*}
Proceeding in this way, we observe that
$$[h_{s},\dots,[h_2,[h_1,[f,g]]]\dots]=(-1)^{s+1}\left(\prod_{i=1}^sh_i\right)(w)\cdot \Big(f(w)(\varphi_d^*)^{s+1}(g)-g(w)(\varphi_d^*)^{s+1}(f)\Big).$$
In particular, setting $s=m-1$ gives us 
$$[h_{m-1},\dots,[h_2,[h_1,[f,g]]]\dots]=0$$
which shows that $S^d(V^*)$ is nilpotent. 
\end{proof}

An unexpected application of the bracket product (\ref{bracket}) is that it can provide a new way to construct infinite-dimensional 
 solvable Lie algebras.  We use $S(V^*)$ to denote the symmetric algebra on $V^*$. Then 
 $$S(V^*)=\bigoplus_{d=0}^\infty S^d(V^*)\cong k[x_1,x_2,\dots,x_n]$$
 as infinite dimensional commutative $\N$-graded $k$-algebras. 

\begin{coro}\label{coro2}
Let $f,g\in S(V^*)$ be two homogenous elements. Define
\begin{equation}
\label{ }
[f,g]_s:=\begin{cases}
 0,     & \deg(f)\neq\deg(g), \\
 [f,g],     & \deg(f)=\deg(g),
\end{cases}
\end{equation}
where $[f,g]$ is given by the bracket product (\ref{bracket}). 
Then $S(V^*)$, together with the bracket product $[-,-]_s$  is a solvable infinite dimensional  Lie algebra. 
\end{coro}

\begin{conv}{\rm
We write $\g_{\varphi}^w(d)$ for the Lie algebra on $S^d(V^*)$ defined by the bracket (\ref{bracket}).
If $d$ is fixed, we also write $\g_{\varphi}^w$ for  $\g_{\varphi}^w(d)$. Throughout this article, all group actions or modules are assumed to be left, and all vector spaces are finite-dimensional. 
\hbo}\end{conv}

\section{Examples}
\setcounter{equation}{0}
\renewcommand{\theequation}
{3.\arabic{equation}}
\setcounter{theorem}{0}
\renewcommand{\thetheorem}
{3.\arabic{theorem}}

\noindent We present two examples that illustrates the construction method in the previous section. 

\begin{exam}{\rm
Suppose $V$ is a two-dimensional vector space over $k$. We identify $e_1=\cv10$ and $e_2=\cv01$ to form a basis of $V$. Let $a\in k$ and $\varphi$ be the linear map defined by $e_1\mapsto e_1+a\cdot e_2$ and $e_2\mapsto e_2$. 
Thus with respect to the basis $\{e_1,e_2\}$, the matrix  of $\varphi$ on $V$ is
$$\varphi_V=\begin{pmatrix}
    1  &  0  \\
    a  &  1
\end{pmatrix}.$$
On the dual space $V^*$, as all actions are assumed to be left, we also identify $x_1=\cv10$ and $x_2=\cv01$ for the dual basis of $V^*$ where the dual relations between $x_i$ and $e_j$ are defined by the dot product in the column vector space $k^2$.

Since $\varphi(x_1)(e_1)=x_1(\varphi(e_1))=x_1(e_1+a\cdot e_2)=1$ and $\varphi(x_1)(e_2)=x_1(\varphi(e_2))=x_1(e_2)=0$, it follows that $\varphi(x_1)=x_1$. A similar computation shows that $\varphi(x_2)=a\cdot x_1+x_2$. Thus with respect to the basis $\{x_1,x_2\}$, the matrix  of $\varphi$ on $V$ is
$$\varphi_{V^*}=\begin{pmatrix}
    1  &  a  \\
    0  &  1
\end{pmatrix}$$
which is equal to the transpose of $\varphi_V$. Actually,  the conclusion $\varphi_{V^*}=\varphi_{V}^T$
holds for general $n$ and $\varphi$.

Recall that the second symmetric power $S^2(V^*)$ has a basis $\{x_1^2,x_1x_2,x_2^2\}$ and the induced action of $\varphi$
on $S^2(V^*)$ is given by $x_1^2\mapsto x_1^2, x_1x_2\mapsto x_1(a\cdot x_1+x_2)$, and $x_2^2\mapsto (a\cdot x_1+x_2)^2.$
Identifying 
$$x_1^2=\cvv100,x_1x_2=\cvv010,x_2^2=\cvv001,$$
we obtain the matrix of $\varphi$ on $S^2(V^*)$:
$$\varphi_{S^2(V^*)}=\begin{pmatrix}
      1& 2a&a^2   \\
      0&1&a\\
      0&0&1  
\end{pmatrix}.$$

We choose $w=e_2$ as an eigenvector of $\varphi$ associated to the eigenvalue $1$ because $\varphi(e_2)=e_2$. Recall $S^1(V^*)=V^*$ so $\g_{\varphi}^w(1)$ has the following nontrivial relation:
$$[x_1,x_2]=x_2(e_2)\cdot \varphi(x_1)-x_1(e_2)\cdot \varphi(x_2)=x_1$$
which produces the unique nonabelian two-dimensional Lie algebra. Moreover, if we write
$$\ell_1:=x_1^2,\ell_2:=x_1x_2,\ell_3:=x_2^2$$
for a basis of $S^2(V^*)$, then $\g_{\varphi}^w(2)$ has the following nontrivial relations: 
\begin{eqnarray*}
~[\ell_1,\ell_2]&=&\ell_2(e_2)\cdot \varphi(\ell_1)-\ell_1(e_2)\cdot \varphi(\ell_2)=0\\
~[\ell_1,\ell_3]&=&\ell_3(e_2)\cdot \varphi(\ell_1)-\ell_1(e_2)\cdot \varphi(\ell_3)=\ell_1\\
~[\ell_2,\ell_3]&=&\ell_3(e_2)\cdot \varphi(\ell_2)-\ell_2(e_2)\cdot \varphi(\ell_3)=a\cdot \ell_1+\ell_2,
\end{eqnarray*}
which is isomorphic to the solvable Lie algebra $L_a^3$ in \cite[Section 4]{deG05} by the base change $\wt{\ell_1}=-a\cdot\ell_1+\ell_2, \wt{\ell_2}=\ell_2$, and $\wt{\ell_3}=\ell_3$.
\hbo}\end{exam}

\begin{exam}{\rm
Suppose $V$ denotes a three-dimensional vector space over $k$. With a basis $\{e_1,e_2,e_3\}$ of $V$, we assume that a nilpotent map $\varphi:V\ra V$ has the following matrix form:
$$\varphi_V=\begin{pmatrix}
      0& 0&0   \\
      a&0&0\\
      b&c&0 
\end{pmatrix}$$
where $a,b,c\in k$. Suppose $\{x_1,x_2,x_3\}$ denotes the basis of $V^*$ dual to $\{e_1,e_2,e_3\}$. With respect to the basis 
$\{x_1^2,x_1x_2,x_1x_3,x_2^2,x_2x_3,x_3^3\}$ (in this order) of $S^2(V^*)$, the matrix of $\varphi$ on $S^2(V^*)$ is 
$$\begin{pmatrix}
    0  &  A  \\
    0  & B
\end{pmatrix}$$ where
$$A=\begin{pmatrix}
      a^2& ab&b^2   \\
      0&ac&2bc\\
      0&0&0 
\end{pmatrix}\textrm{ and }B=\begin{pmatrix}
      0& 0&c^2   \\
      0&0&0\\
      0&0&0 
\end{pmatrix}.$$
Let $w=e_1$ be an eigenvector of $\varphi$ associated to the eigenvalue $0$ and define
$$z_1:=x_1^2,z_2:=x_1x_2,z_3:=x_1x_3,z_4:=x_2^2,z_5:=x_2x_3,z_6:=x_3^3.$$
Then $\g_{\varphi}^w(2)$ is a six-dimensional nilpotent Lie algebra defined by the following nonzero generating relations:
\begin{eqnarray*}
[z_4,z_1] & = & a^2\cdot z_1 \\
~[z_5,z_1] & = & ab\cdot z_1+ac\cdot z_2 \\
~[z_6,z_1]&=&b^2\cdot z_1+2bc\cdot z_2+c^2\cdot z_4.
\end{eqnarray*}
Choosing different values for $a,b,c$ might obtain distinct nilpotent Lie algebras. For instance, setting
$a=b=0$ and $c=1$ gives us a nilpotent Lie algebra isomorphic to $L_{6,2}$ in \cite[Section 4]{deG07}. It is also easy to see that setting $a=c=1$ and $b=0$ will obtain a nilpotent Lie algebra not isomorphic to $L_{6,2}$.
Note that if $\cha(k)\neq 2$, all six-dimensional nilpotent Lie algebras have been classified  in \cite{deG07}.
\hbo}\end{exam}

\begin{rem}{\rm
If $\varphi=0$ or $w=0$, then $\g_{\varphi}^w(d)$ becomes an abelian Lie algebra for all $d\in\N^+$. However, we will see in Proposition \ref{prop5.6} below that neither $\varphi$ and $w$ are not zero, and $\g_{\varphi}^w(d)$ also might be abelian.
\hbo}\end{rem}

\section{Isomorphic Lie Algebras}
\setcounter{equation}{0}
\renewcommand{\theequation}
{4.\arabic{equation}}
\setcounter{theorem}{0}
\renewcommand{\thetheorem}
{4.\arabic{theorem}}

\noindent This section explores  a sufficient condition for when two Lie algebras $\g_{\varphi}^w(d)$ and $\g_{\phi}^v(d)$ are isomorphic.

\begin{thm}\label{thm2}
Let $(\varphi,w)$ and $(\phi,v)$ be two pairs of linear maps and eigenvectors on $V$.
If there exists an invertible linear map $T\in\GL(V)$ such that $\phi=T\circ\varphi\circ T^{-1}$ and $v=T(w)$, then 
$\g_{\varphi}^w(d)$ and $\g_{\phi}^v(d)$ are isomorphic.
\end{thm}

\begin{proof}
For any $f,g\in S^d(V^*)$, we write $T_d^*$ for the linear map on $S^d(V^*)$ induced by $T$ and note that
\begin{eqnarray*}
[T_d^*(f),T_d^*(g)]_w & = & T_d^*(g)(w)\cdot \varphi_d^*(T_d^*(f))- T_d^*(f)(w)\cdot \varphi_d^*(T_d^*(g))\\
 & = & g(T(w))\cdot (T\circ\varphi)_d^*(f)-f(T(w))\cdot (T\circ\varphi)_d^*(g)\\
 &=&g(v)\cdot (\phi\circ T)_d^*(f)-f(v)\cdot (\phi\circ T)_d^*(g)\\
 &=& g(v)\cdot T_d^*(\phi_d^*(f))-f(v)\cdot T_d^*(\phi_d^*(g))\\
 &=&T_d^*\Big(g(v)\cdot \phi_d^*(f)-f(v)\cdot \phi_d^*(g)\Big)=T_d^*([f,g]_v).
\end{eqnarray*}
Thus $T_d^*$ is a Lie homomorphism. As $T$ is bijective, we see that $T_d^*$ is bijective and therefore, 
$T_d^*$ is a Lie isomorphism.
\end{proof}

\begin{coro}\label{coro1}
Let $(\varphi,w)$ be a pair of linear map and eigenvector on $V$ and $0\neq c\in k$ be a scalar. Then 
$\g_{\varphi}^w(d)$ and $\g_{\varphi}^{c\cdot w}(d)$ are isomorphic. 
\end{coro}

\begin{proof}
Setting $T=c\cdot I_V$ where $I_V$ denotes the identity map on $V$ and Theorem \ref{thm2} applies.
\end{proof}

\begin{rem}{\rm This result indicates that for a fixed linear map $\varphi$ on $V$ with the total geometric multiplicity $r>0$, 
all non-isomorphic Lie algebra structures on $S^d(V^*)$, being independent of $d$,  can be parameterized by the $(r-1)$-dimensional projective space $\PP^{r-1}(k)$. 
\hbo}\end{rem}

\begin{coro}
Let $\varphi$ and $\phi$ be two similar linear maps on $V$ and $w$ be an eigenvector of $\varphi$, then there exists
an eigenvector $v$ of $\phi$ such that $\g_{\varphi}^w(d)$ and $\g_{\phi}^v(d)$ are isomorphic.
\end{coro}

\begin{proof}
Suppose $\phi=T\circ \varphi\circ T^{-1}$ for some $T\in\GL(V)$. Setting $v=T(w)$ and together with Theorem \ref{thm2} makes the statement hold.
\end{proof}

\begin{rem}{\rm
Write $\M_d$ for the collection of all possible $\g_{\varphi}^w(d)$, i.e., all Lie algebra structures on $S^d(V^*)$ defined by the bracket product (\ref{bracket}). The consequences above lead us to separate our philosophy of understanding the structure of $\M_d$ into two steps. The first step is to find all similar normal forms for linear maps on $V$; and secondly,
for each normal form $\varphi$, we need to classify those Lie algebras on $S^d(V^*)$ defined by $\varphi$ and its eigenvectors. 
\hbo}\end{rem}

We use the language of matrices to understand $\M_d$. After choosing a basis for $V$, we see that $\M_d$ can be parameterized by the following set:
$$M:=\{(A,w)\in M_n(k)\times k^n\mid A(w)=\lambda \cdot w,\textrm{ for some }\lambda\in k\}$$
where $M_n(k)$ denotes the set of all $n\times n$ matrices over $k$ and $k^n$ denotes the
$n$-dimensional column vector space over $k$.
Clearly, understanding $\M_d$ is equivalent to understanding  $M$. Note that $M$ is an affine variety that
only depends on $n$ and the ground field $k$, not being dependent of $d$.

To better understand $M$ and the orbit spaces under some equivalence relations, we  consider the standard conjugation action of the general linear group $\GL_n(k)$ of all $n\times n$ invertible matrices over $k$ on $M$ defined by
 $$T\cdot (A,w):=(TAT^{-1},T(w))$$
for all $T\in \GL_n(k)$ and $(A,w)\in M$. Theorem \ref{thm2} indicates that if $(B,v)$ and $(A,w)$ are in the same orbit under this action, then $\g_{A}^w(d)$ and $\g_{B}^v(d)$ are two isomorphic Lie algebras in $\M_d$.

\begin{rem}{\rm
We don't guarantee that the isomorphism classes of all Lie algebras in $\M_d$ can be identified with the orbit space $M/\GL_n(k)$.
Actually, even though $(B,v)$ and $(A,w)$ are not in the same orbit, they are probably able to produce isomorphic Lie algebras in $\M_d$; see Example \ref{exam4.7} below.  
\hbo}\end{rem}

\begin{exam}\label{exam4.7}
{\rm
Consider $k=\C$ and $n=2$. Note that the two matrices 
$$\begin{pmatrix}
   1   &0    \\
     0 &  1
\end{pmatrix}\textrm{ and } \begin{pmatrix}
     1 &  1  \\
     0 &  1
\end{pmatrix}$$
are not similar to each other but together with suitable eigenvectors they will produce isomorphic two-dimensional nonabelian 
Lie algebras in $\M_1$, because there is only one two-dimensional nonabelian Lie algebra up to isomorphism, with the nonzero generating relation $[x,y]=x$.
\hbo}\end{exam}

 Theorem \ref{thm2} also has the following immediate consequence concerning
an upper bound for the number of isomorphism classes in $\M_d$. This conclusion should be useful in enumerating this number,  especially for the case $k=\F_q$ is a finite field.

\begin{coro}
The number of non-isomorphic Lie algebras in $\M_d$ is less than or equal to the cardinality of 
the orbit space $M/\GL_n(k)$.
\end{coro}

\section{Complex Two-dimensional Maps}
\setcounter{equation}{0}
\renewcommand{\theequation}
{5.\arabic{equation}}
\setcounter{theorem}{0}
\renewcommand{\thetheorem}
{5.\arabic{theorem}}

\noindent We consider $k=\C$ and $n=2$. By the theory of Jordan normal forms in linear algebra, we see that orbits under the standard conjugation action of $\GL_2(\C)$ on $M_2(\C)$ can be represented by matrices of the following forms
$$\begin{pmatrix}
   \lambda_1   &0    \\
     0 &  \lambda_2
\end{pmatrix}\textrm{ or } \begin{pmatrix}
     \lambda &  1  \\
     0 &  \lambda
\end{pmatrix}$$
where $\lambda_1,\lambda_2,\lambda\in\C$. We write $\{e_1,e_2\}$ for the
standard basis for $V=\C^2$ and $\{x_1,x_2\}$ denotes the basis for $V^*$ dual to $\{e_1,e_2\}$. Then 
$\{y_i:=x_1^{d-i}\cdot x_2^i\mid i=0,1,2,\dots,d\}$ is a basis for $S^d(V^*)$. 

We use $\g_0$ to denote the abelian Lie algebra structure on $S^d(V^*)$.

\begin{prop}\label{prop5.1}
Let $\varphi=\lambda\cdot I_2$ be a scalar matrix in $M_2(\C)$ for any $\lambda\in \C$. Suppose $w\in\C^2$ is any vector. 
Then $\g_{\varphi}^w(d)$ is either abelian or isomorphic to the Lie algebra $\g_1$ with the following generating relations:
\begin{equation}\label{g1}\tag{$\g_1$}
[y_0,y_i]=y_i,
\end{equation}
where $i=1,2,\dots,d.$
\end{prop}

\begin{proof}
Clearly, if $\lambda=0$ or $w=0$, $\g_{\varphi}^w(d)$ is an abelian Lie algebra. Thus we may assume that
$\lambda\neq 0$ and $w\neq 0$. Moreover, taking $w=e_1$, we claim that $\varphi$ and $e_1$ will produce a Lie algebra that is isomorphic to $\g_1$.  In fact, it follows from (\ref{bracket}) that
$$[y_j,y_i]=0$$
for all $j,i\in\{1,2,\dots,d\}$, since $y_j(e_1)=x_1^{d-j}(e_1)\cdot x_2^j(e_1)=0$ for all $j=1,2,\dots,d$. Thus, nonzero generating relations only appear in the brackets $[y_0,y_i]$ where $i=1,2,\dots,d$. Note that
$$[y_0,y_i]=y_i(e_1)\cdot \varphi(y_0)-y_0(e_1)\cdot\varphi(y_i)=-\lambda^d\cdot y_i.$$
As $\lambda\neq 0$, we may set $\wt{y}_0=-\lambda^{-d}\cdot y_0$ and obtain
$$[\wt{y}_0,y_i]=y_i$$
for all $i=1,2,\dots,d$. Hence, the claim holds. 

Now we need to prove that for any nonzero $w\in\C^2$, the Lie algebra $\g_{\varphi}^w(d)$ is isomorphic to $\g_1$. Recall the standard action of $\GL_2(\C)$ on $\C^2$ and restricts on $\C^2\setminus\{0\}$ to be transitive, so there exists an invertible matrix $T\in\GL_2(\C)$
such that $w=T(e_1)$. As $\varphi$ and $T$ commute, it follows that 
$$\varphi=T\cdot\varphi \cdot T^{-1}.$$
By Theorem \ref{thm2}, we see that $\g_{\varphi}^w(d)$ and $\g_1$ are isomorphic.
\end{proof}

\begin{rem}{\rm
The nonzero generating relations of the Lie algebra $\g_1$ also can be determined by the inner derivation $\ad_{y_0}$. With respect to $\{y_0,y_1,\dots,y_d\}$, $\ad_{y_0}$ has the following matrix form:
$$\ad_{y_0}=\dia\{0,1,\dots,1\}.$$
In particular, when $d=2$, the Lie algebra $\g_1$ is three-dimensional solvable, and isomorphic to $L^2$ in \cite[Section 4]{deG05}; and when $d=3$, $\g_1$ is isomorphic to $M^2$ in \cite[Section 5]{deG05}.
\hbo}\end{rem}

\begin{prop}\label{prop5.3}
Let $\varphi=\begin{pmatrix}
   \lambda_1   &0    \\
     0 &  \lambda_2
\end{pmatrix}\in M_2(\C)$ with $\lambda_1\neq\lambda_2$. Suppose $w\in\C^2$ is any nonzero vector. 
Then $\g_{\varphi}^w(d)$ is isomorphic to either the abelian Lie algebra $\g_0$ or the following Lie algebra:
\begin{equation}\label{g2}\tag{$\g_2(c)$}
[y_0,y_i]=c^{i-1}\cdot y_i
\end{equation}
where $i=1,2,\dots,d$, and  $c\in\C\setminus\{0,1\}$.
\end{prop}

\begin{lem} \label{lem5.4}
Let $k$ be any field and $v_1,v_2$ be two linearly independent eigenvectors of a square matrix $A$ associated with two different eigenvalues $\lambda_1$ and $\lambda_2$ respectively. Then $a_1\cdot v_1+a_2\cdot v_2$ can not be an eigenvector of $A$ for all $a_1,a_2\in k^\times$.
\end{lem}

\begin{proof}
Write $v:=a_1\cdot v_1+a_2\cdot v_2$ and assume by the way of contradiction that $v$ is an eigenvector of $A$ associated with 
the eigenvalue $\lambda$. Then $A(v)=A(a_1\cdot v_1+a_2\cdot v_2)=a_1\lambda_1\cdot v_1+a_2\lambda_2\cdot v_2=\lambda\cdot v=a_1\lambda\cdot v_1+a_2\lambda\cdot v_2$. As $v_1$ and $v_2$ are linearly independent, it follows that
$a_1\lambda_1-a_1\lambda=0=a_2\lambda_2-a_2\lambda.$ Since $a_1$ and $a_2$ are nonzero, we see that
$\lambda_1=\lambda=\lambda_2$, which contradicts with the assumption that $\lambda_1\neq\lambda_2$.
\end{proof}

\begin{proof}[Proof of Proposition \ref{prop5.3}]
We first take $w=e_1$ and compute the generating relations of $\g_\varphi^{e_1}(d)$. Observe that $[y_i,y_j]=0$
for all $i,j\in\{1,2,\dots,d\}$ because each $y_i$ involves a positive power of $x_2$ that evaluates on $e_1$ to be zero. Moreover,
$$[y_0,y_i]=y_i(e_1)\cdot \varphi(y_0)-y_0(e_1)\cdot\varphi(y_i)=-\lambda_1^{d-i}\lambda_2^i\cdot y_i$$
for all $i\in\{1,2,\dots,d\}$. As at least one of $\{\lambda_1,\lambda_2\}$ is not zero, we may assume that $\lambda_1\neq 0$
and let $c:=\frac{\lambda_2}{\lambda_1}$. If $\lambda_2=0$, then $\g_\varphi^{e_1}(d)$ is isomorphic to $\g_0$; if $\lambda_2\neq 0$, we may replace $y_0$ by $-\frac{1}{\lambda_1^{d-1}\lambda_2}\cdot y_0$ so that $\g_\varphi^{e_1}(d)$ has the following nonzero generating relations:
$$[y_0,y_i]=c^{i-1}\cdot y_i$$
for all $i\in\{1,2,\dots,d\}$. By Corollary \ref{coro1}, we see that $\g_\varphi^{a\cdot e_1}(d)\cong \g_\varphi^{e_1}(d)\cong\g_2(c)$ for all $a\in\C^\times$.

For the case $w=a\cdot e_2$, we may switch the roles of $e_1$ and $e_2$ and apply a similar argument to see that for all $a\in\C$, $\g_\varphi^{a\cdot e_2}(d)$ is either abelian or isomorphic to $\g_2(c)$.

Now we consider the case $w=a_1\cdot v_1+a_2\cdot v_2$ for some $a_1,a_2\in \C^\times$. By Lemma \ref{lem5.4}, $w$ cannot be an eigenvector of $\varphi$. Hence, in this case, $\g_\varphi^{w}(d)$ doesn't make sense.
\end{proof}

\begin{rem}{\rm
If we extend the range of $c$ in $\g_2(c)$ to $\C^\times$, then $\g_2(1)=\g_1$. Furthermore,
for all $c\in\C\setminus\{0,1\}$, we see that $\g_2(c)$ and $\g_1$ are not isomorphic.
\hbo}\end{rem}

\begin{rem}{\rm
In particular, if $d=2$, then the Lie algebra $\g_2(c)$ is isomorphic to $L^3_a$ in \cite[Section 4]{deG05} for some $a$; in fact, there exists a detailed argument about when $\g_2(c)$ and $L^3_a$ are isomorphic; see \cite[Section 3]{deG05}. 
Moreover, if $d=3$, then $\g_2(c)$ is isomorphic to $M^3_a$ in \cite[Section 5]{deG05} for some $a\in\C$.
\hbo}\end{rem}

\begin{prop}\label{prop5.6}
Let $\varphi=\begin{pmatrix}
   \lambda   &1    \\
     0 &  \lambda
\end{pmatrix}\in M_2(\C)$ for some $\lambda\in\C$ and $w\in\C^2$ be a nonzero vector. 
Then $\g_{\varphi}^w(d)$ is either abelian or isomorphic to  the following Lie algebra:
\begin{equation}\label{g3}\tag{$\g_3(c)$}
[y_0,y_i]=c^{i-1}\cdot \sum_{j=0}^{d-i}{d-i\choose j} c^j\cdot y_{d-j}
\end{equation}
where $i=1,2,\dots,d$ and $c\in\C^\times$.
\end{prop}

\begin{proof}
Since the eigenspace of $\varphi$ associated with $\lambda$ is one-dimensional, we see that
$e_1$ spans the eigenspace and so together with Corollary \ref{coro1} it suffices to consider the case of $w=e_1$.
Note that $\varphi(x_1)=x_2+\lambda\cdot x_1$ and $\varphi(x_2)=\lambda\cdot x_2$. Thus
$$\varphi(y_{i}) = \lambda^{i}\cdot \sum_{j=0}^{d-i}{d-i\choose j} \lambda^j\cdot y_{d-j}$$
for $i=0,1,2,\dots,d$. Since  $[y_i,y_j]=0$ for all $i,j\in\{1,2,\dots,d\}$, it suffices to compute the values of $[y_i,y_0]$
for $i=1,2,\dots,d$. In fact, 
$$[y_0,y_i]=y_i(e_1)\varphi(y_0)-y_0(e_1)\varphi(y_i)=-\varphi(y_i)$$
for all $i\in\{1,2,\dots,d\}$. If $\lambda=0$, then $\g_{\varphi}^{e_1}(d)$ is abelian. Now we may assume that $\lambda\neq 0$
and replace $y_0$ by $-\frac{1}{\lambda}\cdot y_0$. This gives us 
$$[y_0,y_i]=\lambda^{i-1}\cdot \sum_{j=0}^{d-i}{d-i\choose j} \lambda^j\cdot y_{d-j}$$ 
for all $i\in\{1,2,\dots,d\}$. 
\end{proof}

\begin{rem}\label{rem5.8}
{\rm
We also obtain an example 
for which $\varphi=\begin{pmatrix}
  0&1    \\
     0 &0
\end{pmatrix}\neq 0$ and $w=e_1\neq 0$ but $\g_{\varphi}^{w}(d)$ is an abelian Lie algebra. 
\hbo}\end{rem}

Note that $\g_1$ and $\g_2(1)$ are two isomorphic Lie algebras. We may summarize  
the main result in this section as follows. 

\begin{thm}\label{thm3}
Let $n=2$ and $\g$ be a Lie algebra in $\M_d$ over $\C$. Then $\g$ is isomorphic to one of $$\{\g_0,\g_2(c),\g_3(c)\mid c\in\C^\times\}.$$
\end{thm}

\begin{proof}
Combing Propositions \ref{prop5.1}, \ref{prop5.3}, and \ref{prop5.6} together obtains the result.
\end{proof}

We close this article with the following example that illustrates Corollary \ref{coro2} and tells us how to use
a linear map $\varphi$ and its eigenvector $w$ to construct an infinite-dimensional Lie algebra. 

\begin{exam}{\rm
Let's take $\varphi=-I_2$ and $w=e_1$ as  in Proposition \ref{prop5.1} and consider the infinite-dimensional vector space $\C[x_1,x_2]$. Using Corollary \ref{coro2} and Proposition \ref{prop5.1} together gives us a solvable Lie algebra structure on 
$\C[x_1,x_2]$ determined by the following nontrivial bracket products:
$$\left\{[x_1^d,x_1^{d-i}x_2^i]=x_1^{d-i}x_2^i\mid 1\leqslant i\leqslant d,d\in\N^+\right\}.$$
Note that different pairs $(\varphi,w)$ might obtain non-isomorphic Lie algebra structures on $\C[x_1,x_2]$. \hbo
}\end{exam}

\vspace{2mm}
\noindent \textbf{Acknowledgements}. 
This research was partially supported by the Algoma University under grant No. AURF-PT-40370-71.
The author would like to thank the anonymous referees and the editor for their careful reading, constructive comments, and suggestions. Many thanks go to Emmy Chen and Professor Runxuan Zhang for their support.


\raggedright
\end{document}